\documentclass[12pt,letterpaper]{amsart}
\usepackage{amsfonts}
\usepackage{amssymb}
\usepackage{amsthm}
\usepackage{amssymb}
\usepackage[all]{xy}
\usepackage{graphicx}
\usepackage{enumerate}
\usepackage{subfig}
\usepackage{hyperref}
\usepackage{setspace}
\usepackage{bbm}
\def \Sum{\displaystyle\sum}
\newcommand{\norm}[1]{\left\| #1 \right\|}

\newcommand{\supp}{\operatorname{supp}}

\newcommand{\R}{{\mathbb R}}
\newcommand{\C}{{\mathbb C}}

\newcommand{\ip}[2]{\left\langle #1,#2\right\rangle}

\theoremstyle{plain}
\newtheorem{theorem}{Theorem}[section]
\newtheorem{corollary}[theorem]{Corollary}
\newtheorem{lemma}[theorem]{Lemma}

\newtheorem{proposition}[theorem]{Proposition}

\numberwithin{equation}{section}
\numberwithin{figure}{section}

\theoremstyle{definition}
\newtheorem{definition}[theorem]{Definition}
\newtheorem{example}[theorem]{Example}
\newtheorem{remark}[theorem]{Remark}

\begin{document}
\title{Invertibility of graph translation and support of Laplacian Fiedler vectors}

\author{Matthew Begu\'e}

\address{Matthew Begu\'e\\
Department of Mathematics\\
University of Maryland\\
College Park, MD, 20742 USA}

\email{begue@math.umd.edu}

\author{Kasso A.~Okoudjou} 

\address{Kasso A.~Okoudjou\\
Department of Mathematics\\
University of Maryland\\
College Park, MD, 20742 USA}
 \thanks{ K.A.O was partially supported by a grant from the Simons Foundation ($\# 319197$), and ARO grant W911NF1610008.}

\email{kasso@math.umd.edu}



\date{\today}
\maketitle \pagestyle{myheadings} \thispagestyle{plain}
\markboth{M.BEGU\'E AND K.A.OKOUDJOU}{INVERTIBILITY OF GRAPH TRANSLATION AND SUPPORT OF FIEDLER VECTORS}

\begin{abstract}
The graph Laplacian operator is widely studied in spectral graph theory largely due to its importance in modern data analysis.  Recently, the Fourier transform and other time-frequency operators have been defined on graphs using Laplacian eigenvalues and eigenvectors.  We extend these results and prove that the translation operator to the $i$'th node is invertible if and only if all eigenvectors are nonzero on the $i$'th node.  Because of this dependency on the support of eigenvectors we study the characteristic set of Laplacian eigenvectors.  We prove that the Fiedler vector of a planar graph cannot vanish on large neighborhoods and then explicitly construct a family of non-planar graphs that do exhibit this property.
\end{abstract}

\section{Introduction}\label{sec1}
\subsection{Preliminaries}\label{subsec1.1}
Techniques and methods from spectral graph theory and applied and computational harmonic analysis are increasingly being used to analyze, process and make predictions on the huge data sets being generated by the technological advances of the last few decades, e.g., see \cite{pagerank, cameron2014, Kempe2003, Felzenszwalb2004}. At the same time these tasks on large data sets and networks require new mathematical technologies which are leading to  the golden age of Mathematical Engineering  \cite{belkniyog2003, coifmanlafon2006}.

As a result, theories like the vertex-frequency analysis have emerged in an effort to investigate data from both a computational harmonic analysis and spectral graph theoretical point of views \cite{ShumanIEEE}. In particular,  analogues of fundamental concepts and tools such as time-frequency analysis \cite{vertexfrequency}, wavelets \cite{graphwavelets}, sampling theory \cite{AnisOrtega2014}, have been developed in the graph context.  Much of this is done via spectral properties of the graph Laplacian, more specifically through the choice of some eigenbases of the graph Laplacian.  However, and to the best of our knowledge,  a qualitative analysis of the effect of this choice on the resulting theory has not been undertaken. In this paper we consider such qualitative analysis, focusing on the effect of the choice of eigenbasis for the graph Laplacian, on the graph translation operator defined in \cite{vertexfrequency}.

Throughout this paper we shall consider  finite unweighted and undirected graphs. To be specific, a  graph is defined by the pair $(V,E)$ where $V$ denotes the set of vertices and $E$ denotes the set of edges.  When the vertex and edge set $(V,E)$ are clear, we will simply denote the graph by $G$.  We assume that the cardinality of $V$ is $N$. 
 Each element in the edge set $E$ is denoted as an ordered pair $(x,y)$ where $x,y\in V$.  If $(x,y)\in E$, we will often write $x\sim y$ to indicate that vertex $x$ is connected to $y$.  In such case, we say that $y$ is a neighbor of $x$.
 
A graph is \emph{undirected} if the edge set is symmetric, i.e., $(x,y)\in E$ if and only if $(y,x)\in E$.  In the sequel, we only consider  undirected graphs.  The graph is \emph{simple} if there are no self-loops, that is, the edge set contains no edges of the form $(x,x)$.  Additionally, we assume that graphs have at most one edge between any two pair of vertices, i.e., we do not allow multiple-edges between vertices.

The \emph{degree} of vertex $x\in V$ in an undirected graph equals the number of edges emminating from (equivalently, to) $x$ and is denoted $d_x$.  A graph is called \emph{regular} if every vertex has the same degree; it is called $k$-regular when that degree equals $k\in \mathbb{N}$. We refer to  \cite{chungbook} for more background on graphs.

A \emph{path} (of length $m$), denoted $p$, is defined to be a sequence of adjacent edges, $p=\{(p_{j-1},p_j)\}_{j=1}^m$.  We say that the path $p$ connects $p_0$ to $p_m$.  A path is said to be simple if no edge is repeated in it.  A graph is \emph{connected} if for any two distinct vertices $x,y\in V$, there exists some path connecting $x$ and $y$.  

We make the following definition of a ball on a graph that is motivated by the definition of a closed ball on a metric space.
\begin{definition}\label{def:graphball}
Given any $x\in V$ and any integer $r\geq 1$, we define the \emph{ball of radius $r$ centered at $x$}, 
$$B_r(x)=\{y\in V: d(x,y)\leq r\},$$
where $d(x,y)$ is the shorted path length from $x$ to $y$ in $G$.
\end{definition}

\subsection{The graph Laplacian}\label{subsec1.2}

We will consider functions on graphs that take on real (or complex) values on the vertices of the graph.  
Since $V=\{x_i\}_{i=1}^N$ is finite, it is often useful, especially when doing numerical computations, to represent $f:V\to\R$ (or $\C$) as a vector of length $N$ whose $i$'th component equals $f(x_i)$.

Given a finite  graph $G(V,E)$ the adjacency, the \emph{adjacency matrix} is the $N\times N$ matrix, $A$, defined by 
$$A(i,j)=\left\{ \begin{array}{ll} 1, & \text{if }x_i\sim x_j\\  0,&\text{otherwise.}\end{array}\right.$$

The \emph{degree matrix} is the $N\times N$ diagonal matrix $D$ whose entries equal the degrees $d_{x_i}$, i.e., 
$$D(i,j)=\left\{ \begin{array}{ll} d_{x_i}, & \text{if }i=j\\  0,&\text{otherwise.}\end{array}\right.$$

The main differential operator that we shall study is $L$, the Laplacian (Laplace's operator, or graph Laplacian).  
The pointwise formulation of the  Laplacian applied to a function $f:V\to\C$ is given by 
\begin{equation}\label{pointwiseL}
L f(x)=\Sum_{y\sim x} f(x)-f(y).\end{equation}

The graph Laplacian, $L$, can be conveniently represented as a matrix, which, by an abuse of notations, we shall also denote by $L$.    It follows from~\eqref{pointwiseL} that the $(i, j)^{th}$ entry of $L$ is given by
\begin{equation}\label{LaplacianMatrix1}
L(i,j)= \left\{ \begin{array}{l l}  d_{x_i} & \text{if } i=j \\ -1 & \text{if } x_i\sim x_j\\ 0 & \text{otherwise,}\end{array}\right.\end{equation}
or, equivalently, $L=D-A$.  Matrix $L$ is called the unnormalized Laplacian to distinguish it from the normalized Laplacian, $\mathcal{L}=D^{-1/2} L D^{-1/2}= I-D^{-1/2}AD^{-1/2}$, used in some of the literature on graphs, e.g., \cite{chungbook}.  However, we shall work exclusively with the unnormalized Laplacian and shall henceforth just refer to it as the Laplacian.

Furthermore, it is not difficult to see that $$\ip{Lf}{f}=\sum_{x \sim y}|f(x)-f(y)|^2$$ for any $f\in \C^N$.  Consequently, $L$ is a positive semidefinite matrix whose eigenvalues,  $\{\lambda_k\}_{k=0}^{N-1}\subset [0, \infty).$  In addition,  if the graph $G(V, E)$ is  connected the spectrum of the Laplacian $L$  is given by $$0=\lambda_0<\lambda_1\leq \cdots \leq \lambda_{N-1}.$$  
Throughout, we shall denote by $\Phi$ the  set of orthonormal eigenvectors $\{\varphi_k\}_{k=0}^{N-1}$. We abuse notations, and view  $\Phi$ as a  $N\times N$ orthogonal matrix whose $(k-1)$th column is the vector $\varphi_k$.  Note that $\Phi$ is not unique, but for the theory that follows, we assume that one has fixed an eigenbasis and hence the matrix $\Phi$ is assumed to be fixed.

In fact, the following result completely characterizes the relationship between eigenvalues of a graph and connectedness properties of the graph.

\begin{theorem}[\cite{chungbook}]\label{th:constvector}
If the graph $G$ is connected then $\lambda_0=0$ and $\lambda_i>0$ for all $1\leq k \leq N-1$.  In this case $\varphi_0\equiv1/\sqrt{N}$.  More generally, if the graph $G$ has $m$ connected components, then $\lambda_0=\lambda_1=\cdots=\lambda_{m-1}=0$ and $\lambda_k>0$ for all $k=m,...,N-1$.  The indicator function on each connected component (properly renormalized), forms an orthonormal eigenbasis for the $m$-dimensional eigenspace associated to eigenvalue 0.
\end{theorem}

As seen from Theorem~\ref{th:constvector}, the first nonzero eigenvalue of $L$ is directly related to whether or not the graph is connected. In fact, $\lambda_1$ is known as the \emph{algebraic connectivity} of the graph, see \cite{Fiedler1973}, and is widely studied.   Its corresponding eigenvector, $\varphi_1$, is known as the \emph{Fiedler vector} \cite{Fiedler1973, Fiedler1975} and will be discussed more in-depth in Section \ref{ch:eigenvector}. If $\lambda_1$ has multiplicity $1$, then the corresponding Fiedler vector is unique up to a sign.  The Fiedler vector is used extensively in dimensionality reduction techniques \cite{belkniyog2003, coifmanlafon2006, CzajaSchroedinger}, data clustering \cite{newman2003}, image segmentation \cite{felz2004}, and graph drawing \cite{Spielmannotes}.
Finally, we observe that the highest $\lambda_1$ can be is $N$, which happens only for the complete graph in which case the spectrum is $\{0,N,...,N\}$.

\subsection{Outline of the paper}\label{subsec1.3} 
The rest of the paper is organized as follows. In Section \ref{sec:fourier} we review the theory of vertex-frequency analysis on graphs introduced in \cite{vertexfrequency}.  We primarily focus on the graph translation operator since it has substantial differences to the classical Euclidean analogue of translation.  In general, the translation operator is not invertible in the graph setting. We prove when the graph translation operator acts as a semigroup and completely characterize conditions in which the operator is invertible and derive its inverse.

In Section \ref{ch:eigenvector} we investigate  characteristic sets (sets of zeros) of eigenvectors of the Laplacian because it is directly related to the theory of translation developed in Section \ref{sec:fourier}.  In particular, we focus on the support of the Fiedler vector of the graph.  We prove in Section \ref{subsec3.1} that planar graphs cannot have large neighborhoods of vertices on which the Fiedler vector vanishes.  We then introduce a family of (non-planar) graphs, called barren graphs, that have arbitrarily large neighborhoods on which the Fiedler vector does vanish in Section \ref{sec:barren}.  In Section \ref{sec:graphadding} we prove results about the algebraic connectivity and Fiedler vector of a graph formed by adding multiple graphs.

\section{Translation operator on graphs}\label{sec:fourier}
The notions of graph Fourier transform, vertex-frequency analysis, convolution, translation, and modulation operators were recently introduced in \cite{vertexfrequency}.   In this section, we focus on the translation operator and investigate certain of its properties including, semi-group (Theorem~\ref{th:semigroup}),   invertibility and isometry (Theorem~\ref{th:isometry}).

Analogously to the classical Fourier transform on the real line which expressed a function $f$ in terms of the eigenfunctions of the Laplace operator, we define the \emph{graph Fourier transform}, $\hat f$, of a functions $f:V\to\C$ as the expansion of $f$ in terms of the eigenfunctions of the graph Laplacian.  

\begin{definition}
Given the graph, $G$, and its Laplacian, $L$, with spectrum $\sigma(L)=\{\lambda_k\}_{k=0}^{N-1}$ and eigenvectors $\{\varphi_k\}_{k=0}^{N-1}$, the \emph{graph Fourier transform} of $f:V\to\C$ is by 

\begin{equation}\label{graphFT}
\hat f(\lambda_k) = \langle f,\varphi_k\rangle = \Sum_{n=1}^N f(n) \varphi_k^*(n).
\end{equation}
\end{definition}
Notice that the graph Fourier transform is only defined on values of $\sigma(L)$. In particular,  one should interpret the notation $\hat f(\lambda_k)$ to designate the inner product of $f$ with the $k$'th eigenfunction of $L$.  However to emphasize the interplay between the vertex and spectral domains, we shall abuse the notation as defined here.

The \emph{graph inverse Fourier transform} is then given by

\begin{equation}\label{eq:invgraphFT}
f(n)=\Sum_{k=0}^{N-1} \hat f(\lambda_k)\varphi_k(n).
\end{equation}

It immediately follows from the above definition that Parseval's equality holds in this setting as well. Indeed, 
for any $f,g:V\to  \C$, then $\langle f,g\rangle=\langle\hat f,\hat g\rangle$.  Consequently, 
$$\norm{f}_{\ell^2}^2=\Sum_{n=1}^N |f(n)|^2=\Sum_{l=0}^{N-1}|\hat f(\lambda_\ell)|^2=\norm{\hat f}_{\ell^2}^2.$$

Recall that the convolution of two signals $f,g\in L^2(\mathbb R)$ can be defined via the Fourier transform as $\widehat{(f\ast g)}(\xi)=\hat f(\xi)\hat g(\xi)$.  Using this approach, and  by taking the inverse graph Fourier transform, \eqref{eq:invgraphFT}, we can define convolution in the graph domain.  For $f,g:V\to\C$, we define the \emph{graph convolution} of $f$ and $g$ as
\begin{equation}\label{graphconv}
f\ast g(n) = \Sum_{l=0}^{N-1} \hat f(\lambda_\ell) \hat g(\lambda_\ell) \varphi_\ell(n).
\end{equation}

Many of the classical time-frequency properties of the convolution including commutativity, distributivity, and associativity hold for the graph convolution, see {\cite[Proposition 1]{vertexfrequency}}, and all follow directly from the definition of graph convolution \eqref{graphconv}.

For any $k=0,1,..., N-1$ the \emph{graph modulation operator} $M_k:\mathbb C^N\to \mathbb C^N$ is defined as
\begin{equation}\label{eq:genmod}
(M_k f)(n)=\sqrt N f(n) \varphi_k (n).
\end{equation}
Notice that since $\varphi_0 \equiv \frac{1}{\sqrt N}$ then $M_0$ is the identity operator.

An important remark is that in the classical case, modulation in the time domain represents translation in the frequency domain, i.e., $\widehat{M_\xi f}(\omega)=\hat f(\omega-\xi)$.  The graph modulation does not exhibit this property due to the discrete nature of the  spectral domain.  However, it is worthy to notice the special case if $\hat g(\lambda_\ell)=\delta_0(\lambda_\ell)$, i.e., $g$ is a constant function,  then 
\begin{equation*}
\widehat{M_k g}(\lambda_\ell)=\Sum_{n=1}^N \varphi_\ell^*(n)(M_k g)(n) =\Sum_{n=1}^N \varphi_\ell^*(n) \sqrt N \varphi_k(n) \frac{1}{\sqrt N} = \delta_\ell(k).
\end{equation*}
Consequently, if $g$ is the constant unit function, $M_kg=\varphi_k$.

Formally, the translation of a function defined on $\C$ is given by 
$$(T_u f)(t)= f(t-u)= (f\ast \delta_u)(t) .$$
Motivated by this example, for any $f:V\to\mathbb C$ we can define the \emph{graph translation operator}, $T_i: \mathbb C^N\to\mathbb C^N$ via the graph convolution of the Dirac delta centered at the $i$'th vertex:

\begin{equation}\label{eq:graphtrans}
(T_i f)(n) = \sqrt{N} (f\ast \delta_i)(n) = \sqrt{N} \Sum_{k=0}^{N-1} \hat f(\lambda_k)\varphi_k^*(i)\varphi_k(n).
\end{equation}

We can express $T_i f$ in matrix notation as follows:
\begin{equation}T_if=\sqrt N \begin{pmatrix}  
\varphi_0^*(i)\varphi_0(1) & \varphi_1^*(i)\varphi_1(1) & \cdots & \varphi_{N-1}^*(i)\varphi_{N-1}(1)\\
\vdots & \vdots& \cdots & \vdots&\\
\varphi_0^*(i)\varphi_0(N) & \varphi_1^*(i)\varphi_1(N) & \cdots & \varphi_{N-1}^*(i)\varphi_{N-1}(N)
\end{pmatrix}
\begin{pmatrix} \hat f(\lambda_0)\\ \vdots \\ \hat f(\lambda_{N-1})\end{pmatrix}.\label{transmatrix}
\end{equation}

Graph translation exhibits commutative properties, i.e., $T_iT_j f= T_jT_i f$, and distributive properties under the convolution, i.e., $T_i(f\ast g) = (T_i f)\ast g = f\ast (T_i g)$, see \cite[Corollary 1]{vertexfrequency}.  Also, using the definitions of graph convolution, it is elementary to show that for any $i,n\in\{1,...,N\}$ and for any function $g:V\to\mathbb C$ we have
$$T_i g(n)=\overline{T_n\bar g(i)}.$$

Observe that if we choose real-valued eigenfunction in the definition of the graph Fourier transform, then we simply have  
$$T_ig(n)=T_ng(i).$$  These results can be further generalized in the following theorem.

\begin{theorem}\label{th:semigroup}Assume $G$ is a graph whose Laplacian has real-valued eigenvectors $\{\varphi_k\}_{k=0}^{N-1}$.  
Let $\alpha$ be a multiindex, i.e. $\alpha=(\alpha_1,\alpha_2,...,\alpha_K)$ where $\alpha_j\in\{1,...,N\}$ for $1\leq j\leq K$ and let $\alpha_0\in\{1,...,N\}$.  We let $T_\alpha$ denote the composition $T_{\alpha_K}\circ\cdots T_{\alpha_2}\circ  T_{\alpha_1}$.  Then for any $f:V\to\mathbb R$, we have $T_\alpha f(\alpha_0)=T_\beta f(\beta_0)$ where $\beta=(\beta_1,...,\beta_K)$ and $(\beta_0,\beta_1,\beta_2,...,\beta_K)$ is any permutation of $(\alpha_0,\alpha_1,...,\alpha_K)$.
\end{theorem}
\begin{proof}
There exists a bijection between the collection of all possible $T_\alpha f(\alpha_0)$ for $|\alpha|=K$, $1\leq \alpha_0\leq N$, and the space of $(K+1)$-tuples with values in $\{1,...,N\}$.  That is, the map that sends $T_\alpha f(\alpha_0)$ to $(\alpha_0,\alpha_1,....,\alpha_K)$ is a bijection.  This enables us to define an equivalence relation on the space $\{1,...,N\}^{K+1}$.  We write $(a_0,...,a_{K})\cong (b_0,...,b_{K})$ if and only if $T_{a_K}\circ\cdots\circ T_{a_1}f(a_0)=T_{b_K}\circ\cdots\circ T_{b_1}f(b_0)$.

By the commutativity of the graph translation operators, $(a_0,a_1...,a_{K})\cong\sigma_1(a_0,a_1...,a_{K})=(a_1,a_0...,a_{K})$, i.e., $\sigma_1$ is the permutation $(1,2)$.  In general, we write $\sigma_i$ to denote the permutation $(i,i+1)$.  Similarly,  $(a_0,a_1...,a_{K})\cong\sigma_i(a_0,a_1...,a_{K})$ for any $i=2,3,...,K-1$.  

We now have that any permutation $\sigma_i$ for $i=1,...,K-1$ preserves equivalency.  This collection of $K-1$ transpositions allow for any permutation, which completes the proof.
\end{proof}

The graph translation operators are distributive with the convolution and the operators commute among themselves.  However, the niceties end here; other properties of translation on the real line do not carry over to the graph setting.  For example, we do not have the collection of graph translation operators forming a group, i.e., $T_i T_j \neq T_{i+j}$.  In fact, we cannot even assert that the translation operators form a semigroup, i.e. $T_iT_j=T_{i\bullet j}$ for some semigroup operator $\bullet:\{1,...,N\}\times\{1,...,N\}\to\{1,...,N\}$.  The following theorem characterizes graphs which do exhibit a semigroup structure of the translation operators.

\begin{theorem}\label{th:Hadamard}
Consider the graph, $G(V,E)$, with real-valued (resp. complex-valued) eigenvector matrix $\Phi=[\varphi_0 \cdots\varphi_{N-1}]$.  Graph translation on $G$ is a semigroup, i.e. $T_iT_j=T_{i\bullet j}$ for some semigroup operator $\bullet:\{1,...,N\}\times\{1,...,N\}\to\{1,...,N\}$, only if $\Phi=(1/\sqrt{N})H$, where $H$ is a real-valued (resp. complex-valued) Hadamard matrix.
\end{theorem}

\begin{proof}
\begin{enumerate}
\item[(a)]  We first show that graph translation on $G$ is a semigroup, i.e. $T_iT_j=T_{i\bullet j}$ for some semigroup operator $\bullet:\{1,...,N\}\times\{1,...,N\}\to\{1,...,N\}$, if and only if $\sqrt{N}\varphi_k(i)\varphi_k(j)=\varphi_k(i\bullet j)$ for all $l=0,...,N-1$.  By the definition of graph translation, we have,
$$T_iT_jf(n)={N} \Sum_{k=0}^{N-1} \hat f(\lambda_{k})\varphi_{k}^*(j)\varphi_k^*(i)\varphi_k(n)$$ and 
$$T_\ell f(n) = \sqrt{N} \Sum_{k=0}^{N-1} \hat f(\lambda_k)\varphi_k^*(\ell)\varphi_k(n).$$
Therefore, $T_iT_jf=T_{i\bullet j} f$ will hold for any function $f:V\to \R$ if and only if $\sqrt{N}\varphi_k(i)\varphi_k(j)=\varphi_k(i\bullet j)$ for every $k\in \{0,...,N-1\}$.

\item[(b)]  We show next that $\sqrt{N}\varphi_k(i)\varphi_k(j)=\varphi_k(i\bullet j)$ for all $k=0,...,N-1$ only if the eigenvectors are constant amplitude, namely $1/\sqrt{N}$ by the orthonormality of the eigenvectors.  Assume $\sqrt{N}\varphi_k(i)\varphi_k(j)=\varphi_k(i\bullet j)$, which, in particular, implies $\sqrt{N}\varphi_k(i)\varphi_k(i)=\sqrt{N}\varphi_k(i)^2=\varphi_k(i\bullet i)$.

Suppose that there exists $k\in\{0,...,N-1\}$ such that $|\varphi_k(a_1)|>1/\sqrt{N}$ for some $a_1\in\{1,...,N\}$.  Then $\sqrt{N}|\varphi_k(a_1)|^2>|\varphi_k(a_1)|$ and so $a_1\bullet a_1 =a_2$ for some $a_2\in\{1,...,N\}$.  Note that $a_2\neq a_1$, if not $|\varphi_k(a_1)|=\sqrt{N}|\varphi_k(a_1)|^2$  which is impossible. Thus,  $a_2\in\{1,...,N\}\setminus \{a_1\}$. Then since $|\varphi_k(a_2)|=\sqrt{N}|\varphi_k(a_1)|^2>|\varphi_k(a_1)|>1/\sqrt{N}$ we can repeat the same argument to assert $a_2\bullet a_2=a_3$ for some $a_3\in\{1,...,N\}\setminus\{a_1,a_2\}$.  Repeat this procedure $N$ times give $|\varphi_k(n)|>1/\sqrt{N}$ for all vertices $n\in\{1, ..., N\}$, which contradicts the notion that $\|\varphi_k\|=1$.

Therefore we have shown that the graph translation operator forms a semigroup only if $|\varphi_k(n)|\leq 1/\sqrt{N}$ for all $n=1, ..., N$ and $k=0, ..., N=1$.  But again, since each eigenvector must satisfy $\|\varphi_k\|=1$, we can strengthen this condition to $|\varphi_k(n)|= 1/\sqrt{N}$ for all $n=1, ..., N$ and $k=0, ..., N=1$.  Since $\Phi$ is an orthogonal matrix, i.e. $\Phi\Phi^*=\Phi^*\Phi=I$, then $\Phi=(1/\sqrt{N}) H$, where $H$ is a Hadamard matrix.
\end{enumerate}
\end{proof}

\begin{remark}
\begin{enumerate}
\item[(a)]  If we relax the constraint that $\Phi$ must be real-valued, we can obtain graphs with constant-amplitude eigenfunctions that allow the translation operators to form a (semi)group.  For the cycle graph on $N$ nodes, $C_N$, one can choose $\Phi$ equal to the discrete Fourier transform (DFT) matrix, where $\Phi_{nm}=e^{-2\pi i (n-1)(m-1)/N}$ .  Under this construction, we have $T_iT_j=T_{i+j \pmod N}$.

\item[(b)]  It is shown in \cite[Theorem 5]{BarFalKir2011} that if $\Phi=(1/\sqrt{N})H$ for Hadamard $H$, then the spectrum of the Laplacian, $\sigma(L)$, must consist entirely of even integers.  The authors of \cite{Fallat2005} explore graphs with integer spectrum but do not address the case of a spectrum of only even integers.

\item[(c)]  The converse to Theorem \ref{th:Hadamard} does not necessarily hold.  That is, if the eigenvector matrix $\Phi=1/\sqrt{N} H$, for a renormalized Hadamard matrix $H$, then the translation operators on $G$ need not form a semigroup.  For example, consider the real Hadamard matrix, $H$, of order 12 given by 
$$ H=\left[ \begin{array}{rrrrrrrrrrrr}
  1     &1    & 1    & 1  &   1   &     1 &    1&     1 &    1 &    1  &   1 &    1\\
     1  &  -1 &    1 &   -1 &    1&     1 &    1&    -1&    -1&    -1 &    1&    -1\\
     1  &  -1 &   -1 &    1 &   -1&     1 &    1&     1&    -1&    -1 &   -1&     1\\
     1  &   1 &   -1 &   -1 &    1&    -1 &    1&     1&     1&    -1 &   -1&    -1\\
     1  &  -1 &    1 &   -1 &   -1&     1 &   -1&     1&     1&     1 &   -1&    -1\\
     1  &  -1 &   -1 &    1 &   -1&    -1 &    1&    -1&     1&     1 &    1&    -1\\
     1  &  -1 &   -1 &   -1 &    1&    -1 &   -1&     1&    -1&     1 &    1&     1\\
     1     &1 &   -1 &   -1 &   -1&     1 &   -1&    -1&     1&    -1 &    1&     1\\
     1    & 1 &    1 &   -1 &   -1&    -1 &    1&    -1&    -1&     1 &   -1&     1\\
     1   &  1 &    1 &    1 &   -1&    -1 &   -1&     1&    -1&    -1 &    1&    -1\\
     1  &  -1 &    1 &    1 &    1&    -1 &   -1&    -1&     1&    -1 &   -1&     1\\
     1 &    1 &   -1 &    1 &    1&     1 &   -1&    -1&    -1&     1 &   -1&    -1
     
     \end{array}\right].
$$

Then the second and third columns multiplied componentwise equals the vector
$$[1,	-1,	1	,-1,	-1,	1	,1,	-1,	1,	1,	-1,	-1]^\top,$$
which does not equal any of the columns of $H$.

\item[(d)]  What kinds of graphs have a Hadamard eigenvector matrix?  The authors of \cite{BarFalKir2011} prove that if $N$ is a multiple of 4 for which a Hadamard matrix exists, then the complete graph on $N$ vertices, $K_N$, is one such graph.
\end{enumerate}

\end{remark}

Unlike in the classical case in $\R^d$, graph translation is not an isometric operation, i.e., $\norm{T_if}_2\neq \norm{f}_2$.  However, {\cite[Lemma 1]{vertexfrequency}} provides the following estimates:

\begin{equation}\label{eq:transestimate}
|\hat f(0)| \leq \norm{T_i f}_2\leq \sqrt{N}  \max_{k\in\{0,1,...,N-1\}} |\varphi_\ell(i)|   \norm{f}_2 \leq \sqrt{N}  \max_{k\in\{0,1,...,N-1\}} \norm{\varphi_\ell}_\infty   \norm{f}_2
\end{equation}

Furthermore, unlike the Euclidean notion of translation, graph translation need not be invertible.  Theorem \ref{th:isometry} characterizes all graphs for which the operator $T_i$ is not invertible.  Additionally, Hadamard matrices appear again in characterizing when graph translation does act as a unitary operator.

\begin{theorem}\label{th:isometry}
The graph translation operator $T_i$  fails to be invertible if and only if there exists some $k=1,...,N-1$ for which $\varphi_k(i)=0$.  In particular, the nullspace of $T_i$ has a basis equal to those eigenvectors that vanish on the $i$'th vertex.  

Additionally, $T_i$ is unitary if and only if $|\varphi_k(i)|=1/\sqrt{N}$ for all $k=0,1,...,N-1$ and all graph translation operators are unitary if and only if $\sqrt{N}\Phi$ is a Hadamard matrix.
\end{theorem}

\begin{proof}

By \eqref{transmatrix}, the operator $T_i$ can be written as the matrix
\begin{eqnarray}T_i&=&\sqrt{N} \begin{pmatrix}  
\varphi_0^*(i)\varphi_0(1) & \varphi_1^*(i)\varphi_1(1) & \cdots & \varphi_{N-1}^*(i)\varphi_{N-1}(1)\\
\vdots & \vdots& \cdots & \vdots&\\
\varphi_0^*(i)\varphi_0(N) & \varphi_1^*(i)\varphi_1(N) & \cdots & \varphi_{N-1}^*(i)\varphi_{N-1}(N)
\end{pmatrix}\Phi^* \notag\\
&=:&\sqrt{N}A_i \Phi^*\label{Aimatrix}
\end{eqnarray}

We can compute the rank of $T_i^*T_i=N\Phi A_i^*A_i\Phi^*$.  Since $\Phi$ is an $N\times N$ matrix of full rank, we can express the rank of $T_i$ solely in terms of the matrix $A_i$, i.e.,
$$\operatorname{rank}(T_i)=\operatorname{rank}(T_i^*T_i)=\operatorname{rank}(\Phi A_i^*A_i\Phi^*)=\operatorname{rank}(A_i^*A_i).$$  

We can explicitly compute for any indices $n,m\in\{1,...,N\}$,
\begin{eqnarray*}
(A_i^*A_i)(n,m)=\Sum_{k=1}^N \overline{A_i(k,n)}A_i(k,m)&=&\Sum_{k=0}^{N-1} \varphi_n^*(k)\varphi_n(i)\varphi_m(k)\varphi_m^*(i)\\
&=& \varphi_n(i)\varphi_m^*(i)\Sum_{k=0}^{N-1}\varphi_n^*(k)\varphi_m(k)\\
&=& \varphi_n(i)\varphi_m^*(i)\delta_n(m).
\end{eqnarray*}
Hence, $A_i^*A_i$ is a diagonal matrix with diagonal entries $(A_i^*A_i)(n,n)= |\varphi_n(i)|^2$.
Therefore,
\begin{equation}\label{Ti*Ti}
T_i^*T_i=N\sum_{k=0}^{N-1}|\varphi_{k}(i)|^2\varphi_k\otimes \varphi_{k}^{*}.
\end{equation}

Consequently, this proves that $\operatorname{rank}(T_i)=|\{k:\varphi_k(i)\neq 0\}|$ and hence $T_i$ is invertible if and only if $\varphi_k(i)\neq 0$ for all $k$.

 Furthermore, $T_iT_i^{*}=\tfrac{N}{N}\sum_{k=0}^{N-1}\varphi_k\otimes \varphi_{n}^{*}=I$ if and only if $|\varphi_n(i)|=1/\sqrt{N}$ for all $n=0,1,...,N-1$.  

Suppose now that $\varphi_{k_j}(i)=0$ for $\{k_j\}_{j=1}^K\subseteq\{1,...,N-1\}$. Hence, $\operatorname{rank}(T_i)=N-K$.  Then for each $j\in\{1,...,K\}$ and any $n\in\{1,...,N\}$ we have 
$$T_i\varphi_{k_j}(n)=\sqrt{N}\Sum_{k=0}^{N-1}\hat\varphi_{k_j}(\lambda_k)\varphi_k^*(i)\varphi_k(n)=\sqrt{N} \varphi_{k_j}^*(i)\varphi_{k_j}(n)=0.$$
Therefore, $\varphi_{k_j}$ is in the null space of $T_i$ for every $j=1,...,K$.  Thus $\{\varphi_{k_j}\}_{j=1}^K$ is a collection of $K$ orthogonal unit-norm vectors in the null space which has dimension $N-\operatorname{rank}(T_i)=K$, hence they form an orthonormal basis for the null space of $T_i$ which proves the claim about the null space of $T_i$.

Finally, if $\sqrt{N}|\varphi_n(i)|=1$ for all $n,i=1,...,N$ then $\sqrt{N}\Phi$ is Hadamard which concludes proof.

\end{proof}

\begin{corollary}\label{cor:Tiinvertible}
If $\varphi_k(i)\neq0$ for all $k=1,...,N-1$, then the graph translation operator $T_i$ is invertible and its inverse is given by 
$$T_i^{-1}=\frac{1}{\sqrt N} \Phi\begin{pmatrix}  
\varphi_0^*(1)  {\varphi_0^*(i)}^{-1}& \varphi_0^*(2) {\varphi_0^*(i)}^{-1}& \cdots & \varphi_0^*(N) {\varphi_0^*(i)}^{-1}\\
\vdots & \vdots& \cdots & \vdots&\\
\varphi_{N-1}^*(1)  {\varphi_{N-1}^*(i)}^{-1}& \varphi_{N-1}^*(2)  {\varphi_{N-1}^*(i)}^{-1} &\cdots & \varphi_{N-1}^*(N)  {\varphi_{N-1}(i)}^{-1}\\
\end{pmatrix}.$$
\end{corollary}
\begin{proof}
We shall first prove that the inverse to the matrix $A_i$ given in \eqref{Aimatrix} is given by 
$$A_i^{-1}=\begin{pmatrix}  
\varphi_0^*(1)  {\varphi_0^*(i)}^{-1}& \varphi_0^*(2) {\varphi_0^*(i)}^{-1}& \cdots & \varphi_0^*(N) {\varphi_0^*(i)}^{-1}\\
\vdots & \vdots& \cdots & \vdots&\\
\varphi_{N-1}^*(1)  {\varphi_{N-1}^*(i)}^{-1}& \varphi_{N-1}^*(2)  {\varphi_{N-1}^*(i)}^{-1} &\cdots & \varphi_{N-1}^*(N)  {\varphi_{N-1}(i)}^{-1}\\
\end{pmatrix}.$$
We can then compute 
\begin{eqnarray*}
A_iA_i^{-1}(n,m)&=&\Sum_{k=1}^N A_i(n,k) A_i^{-1}(k,m)=\Sum_{k=0}^{N-1} \varphi_k^*(i)\varphi_k(n) \varphi_k^*(m){\varphi_k^*(i)}^{-1}\\
&=&\Sum_{k=0}^{N-1} \varphi_k(n)\varphi_k^*(m)=\delta_n(m),
\end{eqnarray*}
and similarly

\begin{eqnarray*}
A_i^{-1}A_i(n,m)&=&\Sum_{k=1}^N A_i^{-1}(n,k) A_i(k,m)=\Sum_{k=1}^N \varphi_{n-1}^*(k)\varphi_{n-1}^*(i)^{-1}\varphi_{m-1}(k)\varphi_{m-1}^*(i)\\
&=&\varphi_{n-1}^*(i)^{-1}\varphi_{m-1}^*(i)\Sum_{k=1}^N \varphi_{n-1}^*(k)\varphi_{m-1}(k)=\varphi_{n-1}^*(i)^{-1}\varphi_{m-1}^*(i)\delta_n(m),
\end{eqnarray*}
which proves $A_i^{-1}A_i=A_iA_i^{-1}=I_N$.

Thus we can verify by the orthonormality of $\Phi$ that 
$$T_iT_i^{-1}=A_i\Phi^*\Phi A_i^{-1}=I_N=\Phi A_i^{-1}A_i\Phi^*=T_i^{-1}T_i.$$
\end{proof}

Since the invertibility of the graph translation operators depends entirely on when and where eigenvectors vanish, Section \ref{ch:eigenvector} is devoted to studying the support of graph eigenvectors.

\begin{remark}
The results of Theorem \ref{th:isometry} and its corollary are not applicable solely to the graph translation operators.  They can be generalized to a broader class of operators on graphs, in particular, operators that act as Fourier multipliers.  An operator $A$ is a Fourier multiplier with symbol $a$ if 
$$\widehat{Af}(\xi)=\hat a(\xi) \hat f(\xi)$$
for some function $a$ defined in the spectral domain.

Indeed graph translation is defined as a Fourier multiplier since it is defined as
$$\widehat{T_i f}(\lambda_k)=\varphi_k(i)\hat f(\lambda_k).$$
Hence, Theorem \ref{th:isometry} and Corollary \ref{cor:Tiinvertible} can be generalized to Fourier multipliers in the following way
\begin{corollary}\label{cor:multipliers}
Let $A$ be a Fourier multiplier whose action on $f:V\to\C$ is defined in the spectral domain $\widehat{Af}(\lambda_k)=\hat a(\lambda_k) \hat f(\lambda_k)$.  Then $A$ is invertible if and only if $\hat a(\lambda_k)\neq 0$ for all $\lambda=0,1,...,N-1$.  Furthermore, its inverse $A^{-1}$ will be given by the Fourier multiplier
$$\widehat{A^{-1}f}(\lambda_k)=\hat a(\lambda_k)^{-1} \hat f(\lambda_k).$$
\end{corollary}

\end{remark}

\section{Support of Laplacian Fiedler vectors on graphs}\label{ch:eigenvector}

This section proves results about the support of Laplacian eigenvectors on graphs.  In particular, we characterize and describe the set on which eigenvectors vanish.  The Fiedler vector, $\varphi_1$, has unique properties that enable us to prove our main result, Theorem \ref{th:planarballs}, that planar graphs cannot have large regions on which $\varphi_1$ vanishes.  We then construct a family of (non-planar) graphs, called the barren graphs, and prove in Theorem \ref{th:barren} that their Fiedler vectors do vanish on large regions. As seen in Theorem \ref{th:isometry}, the support of eigenvectors will influence the behavior of the graph translation operators defined in the last section.

\subsection{The characteristic set of the Fiedler vector}\label{subsec3.1}
Let $\varphi_1$ denote a Fiedler vector for $L$ on $G$.  We can decompose the vertex set, $V$, into three disjoint subsets, $V=V_+\cup V_-\cup V_0$, where $V_+=\{x\in V: \varphi_1(x)>0\}$, $V_-=\{x\in V: \varphi_1(x)<0\}$, and $V_0=\{x\in V: \varphi_1(x)=0\}$.  The set $V_0$, the set of vertices on which the Fiedler vector vanishes, is referred to in literature as the \emph{characteristic set} of the graph \cite{Bapat1998}.  This vertex decomposition is not a unique property to the graph $G$; any graph can allow multiple such decompositions of the vertex set $V$.  In the case that the algebraic connectivity has higher multiplicities, i.e., $\lambda_1=\lambda_2=\cdots=\lambda_m$ for some $2\leq m \leq N-1$, then each $\varphi_s$ is a Fiedler vector for $1\leq s \leq m$.  Futhermore, any linear combination of $\{\varphi_s\}_{s=1}^m$ will also be a Fiedler vector and yield a different vertex decomposition.  Even in the case when the algebraic connectivity of $G$ is simple, then $-\varphi_1$ is also a Fiedler vector for $G$.  In this case, $V_+$ and $V_-$ can be interchanged but the set $V_0$ is unique to $G$. 

We wish to describe and characterize the sets $V_+$, $V_-$, and $V_0$ for graphs.
Fiedler proved in \cite{Fiedler1975} that the subgraph induced on the vertices $\{v\in V : \varphi_1(v)\geq 0\}=V_+\cup V_0$ forms a connected subgraph of $G$.  Similarly, $V_-\cup V_0$ form a connect subgraph of $G$.  Recently, it was proved in \cite{Urschel2014} that we can relax the statement and show that the subgraphs on $V_+$ and $V_-$ are connected subgraphs of $G$.

The following result guarantees that $V_+$ and $V_-$ are always close in terms of the shortest path graph distance.
\begin{lemma}\label{lem:Zset}
Let $G(V,E,\omega)$ with Fiedler vector $\varphi_1$ inducing the partition of vertices $V=V_+\cup V_- \cup V_0$.  Then $d(V_+,V_-)\leq 2$.  
\end{lemma}
\begin{proof}
First consider the case in which $V_0=\emptyset$.  In this case, there necessarily exists an edge $e=(x,y)$ with $x\in V_+$ and $y\in V_-$ and hence $d(V_+,V_-)=1$.

Now consider the case in which $V_0\neq\emptyset$.  Since $G$ is connected we are guaranteed the existence of some $x\in V_0$ and some $y\sim x$ with either $y\in V_+$ or $y\in V_-$.  Since $x\in V_0$, we have 
\begin{equation}
0=\lambda_1 \varphi_1(x)=L\varphi_1(x)=\Sum_{z\sim x} \varphi_1(z)-\varphi_1(x) = \Sum_{z\sim x} \varphi_1(z). \label{Zneighb}
\end{equation} 
Therefore, \eqref{Zneighb} implies the existence of  at least one other neighbor of $x$, call it $y'$, such that $\varphi_1(y')$ has the opposite sign of $\varphi_1(y)$.  Hence we have now constructed a path, namely $(y,x,y')$ connecting $V_+$ and $V_-$ and the lemma is proved.
\end{proof}

Many graphs exhibit the property that eigenvectors $\varphi_k$ for large values of $k$ are highly localized and vanish on large regions of the graph; see \cite{StrichartzLocalized} for an experimental excursion on this phenomenon.  It is perhaps a misconception that eigenvectors corresponding to small eigenvalues, or in particular, the Fiedler vector of graphs have full support.  Indeed the Fiedler vector of the Minnesota graph never achieves value zero.  On the other hand the Fiedler vector of the the graph approximations to the Sierpeinski gasket SG$_n$ can vanish but only along the small number of vertices symmetrically in the center of the graph.  

It was shown in \cite{Bapat1998}, that the cardinality of $V_0$ can be arbitrarily large.  Figure \ref{fig:largeV0} shows a family of graphs that yield sets $V_0$ with arbitrarily large cardinality.  The family is a path graph $P_N$ on an odd number of vertices, except the middle vertex and its edges are duplicated an arbitrarily large number of times.  As evident from Figure \ref{fig:largeV0}, the set $V_0$ is not connected; in fact, no vertex in $V_0$ is connected to any other vertex of $V_0$.

\begin{figure}
\begin{center}
{\includegraphics[scale=.5]{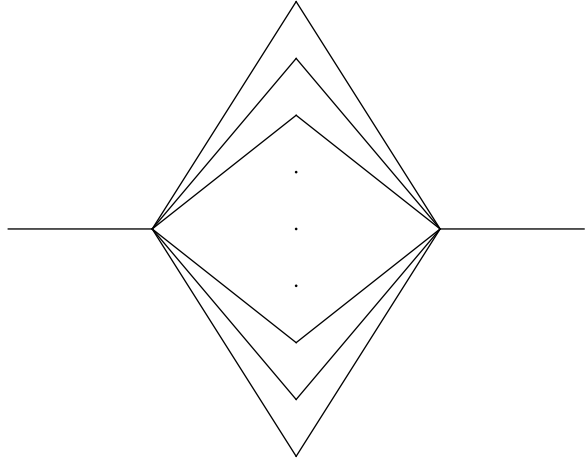}}
\end{center}
\caption{A graph with arbitrarily large set $V_0(\varphi)$}\label{fig:largeV0}
\end{figure}

For the sake of thoroughness we introduce a family of graphs also with arbitrarily large $V_0$ but that is also connected.  We call the family of graphs the \emph{generalized ladder graphs}, denoted $\operatorname{Ladder}(n,m)$.  The standard ladder graphs, Ladder$(n,2)$, is simply the graph Cartesian product, see \cite{imrich2008}, of the path graph of length $n$, $P_n$, and the path graph of length 1, $P_1$.  The graph Ladder$(n,2)$ resembles a ladder with $n$ rungs.  The generalized ladder graphs, Ladder$(n,m)$, are ladders with $n$ rungs and each rung contains $m$ vertices.  Provided that the number of rungs, $n$, is odd, then $V_0$ will be the middle rung and will clearly be connected.  This gives $|V_0|=m$.  Figure \ref{fig:genladder} shows a generalized ladder graph and its Fiedler vector.

\begin{figure}[htbp]
\begin{center}
\includegraphics[scale=.4]{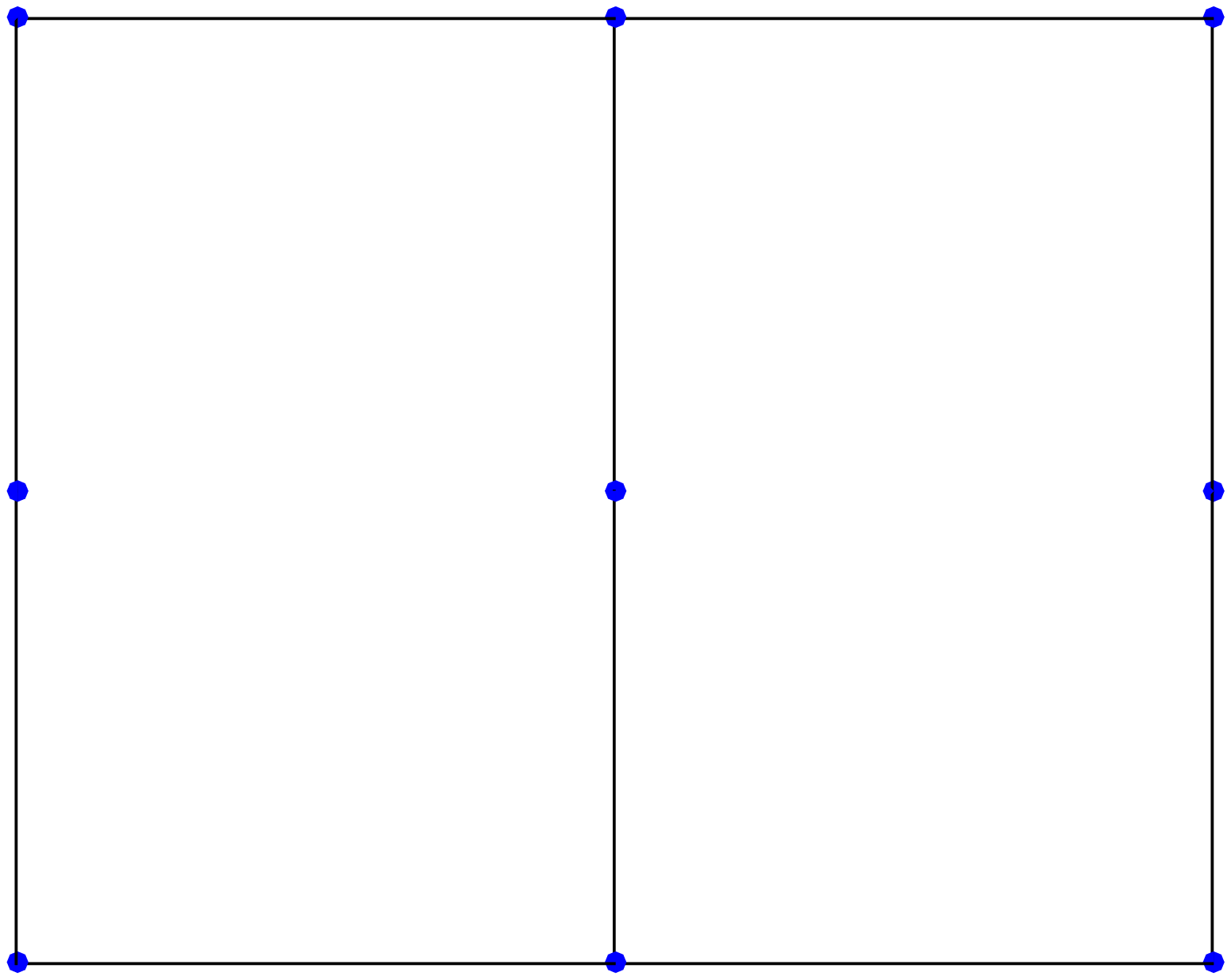}\includegraphics[scale=.4]{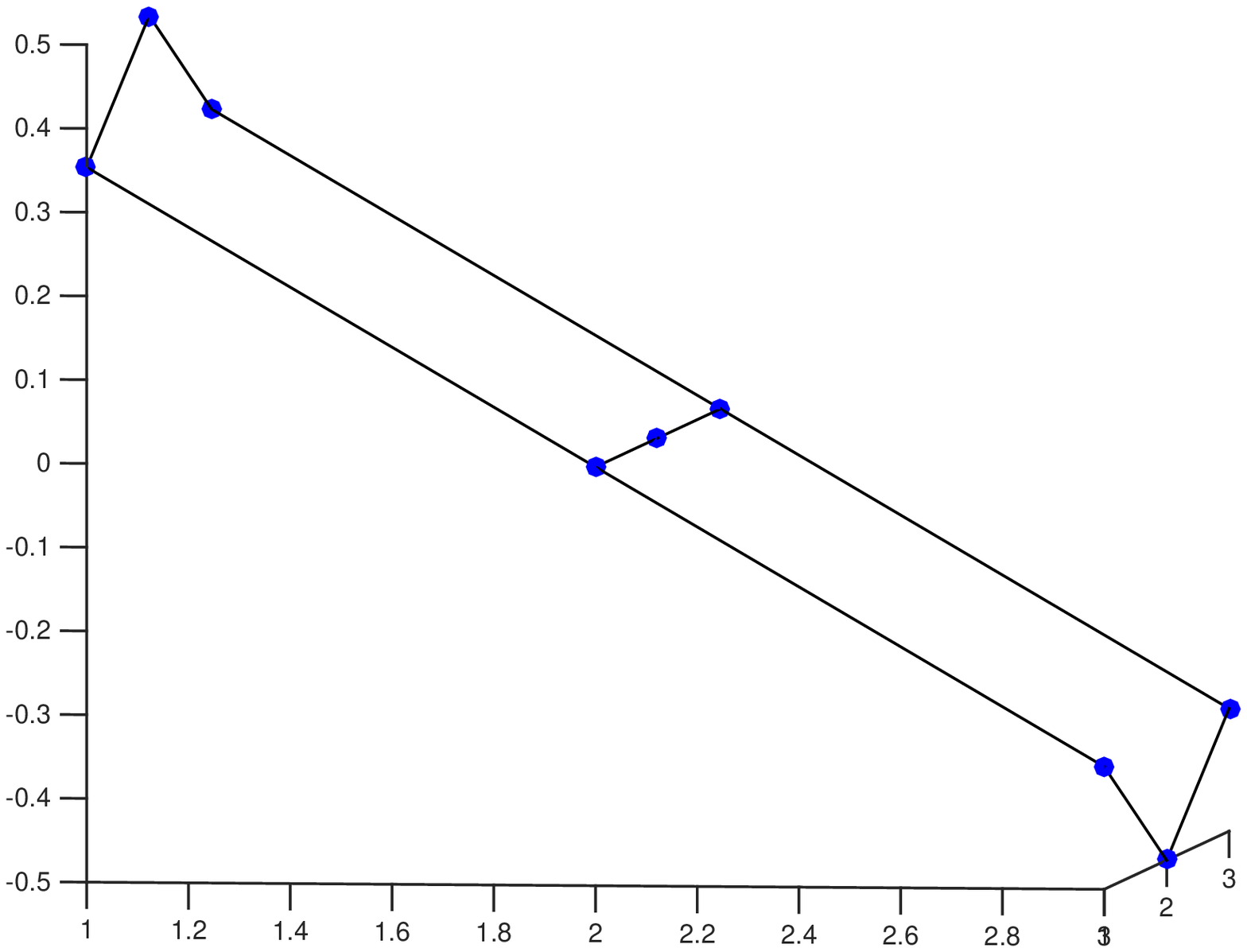}
\end{center}
\caption{Left:  The generalized ladder graph, Ladder$(3,3)$.  Right: A 3-dimensional plot of the Fideler vector $\varphi_1$ on Ladder$(3,3)$.  The set $V_0$ consists of the three vertices making the middle rung of the ladder and contains one ball of three vertices.}\label{fig:genladder}
\end{figure}

The generalized ladder graph provides an example of a graph with a connected characteristic set.  Observe in Figure \ref{fig:genladder} however, that each vertex in $V_0$ is connected to at most two vertices.  We then pose the question as to whether or not there exist graphs for which a vertex $V_0$ has three or more neighbors all contained in $V_0$.  It is simpler to state this property using the definition of a graph ball, given in Definition \ref{def:graphball}.

The following proposition shows that the Fiedler vector cannot be constant-valued on any balls within $V_+$ and $V_-$.
\begin{proposition}\label{prop:flatballs}
Let $\varphi_1$ be the Fiedler vector for the Laplacian on graph $G$ and suppose $B_r(x)\subseteq V_+$ or $B_r(x)\subseteq V_-$ for $r\geq 1$.  Then $\varphi_1$ cannot be constant-valued on $B_r(x)$.
\end{proposition}
\begin{proof}
It suffices to prove the claim for $r=1$.  Without loss of generality, assume $B_1(x)\subseteq V_+$ and suppose that $\varphi_1$ is constant on $B_1(x)$.  Then 
$$L\varphi_1(x)=\Sum_{y\sim x} \omega_{x,y} (\varphi_1(x)-\varphi_1(y))=0,$$ since $y\sim x$ implies $y\in B_1(x)$ and $\varphi_1$ is constant on that ball.  However, $L \varphi_1(x)=\lambda_1\varphi_1(x)>0$ since $\lambda_1>0$ and $\varphi_1(x)>0$ on $V_+$.  This is a contradiction and the proof is complete.
\end{proof}

The result of Proposition \ref{prop:flatballs} can be formulated in terms on \emph{any} non-constant eigenvector of the Laplacian, not just a Fiedler vector.
\begin{corollary}
Any non-constant eigenvector of the Laplacian, $\varphi_k$, associated with eigenvalue $\lambda_k>0$ cannot be constant on any ball contained in the positive vertices $\{i\in V: \varphi_k(i)>0\}$ or negative vertices $\{i\in V: \varphi_k(i)<0\}$ associated to that eigenvector.
\end{corollary}
\begin{proof}
Suppose there existed a ball $B_1(x)\subseteq \{i\in V: \varphi_k(i)>0\}$ on which $\varphi_k$ was constant.  Then just as in the previous proof we could calculate 
$$L\varphi_k(x)=\Sum_{y\sim x} \omega_{x,y} (\varphi_k(x)-\varphi_k(y))=0,$$
which contradicts $L\varphi_k(x)=\lambda_k\varphi_k(x)>0$.
\end{proof}

We wish to extend Proposition \ref{prop:flatballs} to the set $V_0$.  However, as seen in generalized ladder graphs, Ladder$(n,m)$ for $n$ odd and $m>2$, for which $V_0$ contains a ball of radius 1.  This ball, however, contains 3 vertices (the center vertex and its two neighbors).  The next goal is to characterize graphs whose characteristic set $V_0$ contains a ball of radius 1 containing at least four vertices.  We prove that this is impossible for planar graphs.  

\begin{definition}
A \emph{planar graph} is a graph whose vertices and edges can be embedded in $\R^2$ with edges intersecting only at vertices.\end{definition} 

In 1930, Kazimierz Kuratowski characterized all planar graphs in terms of subdivisions.
\begin{definition}
A \emph{subdivision} of a graph $G(V,E)$, also referred to as an \emph{expansion}, is the graph $H(\tilde V, \tilde E)$ where the vertex set is the original vertex set with an added vertex, $w$, and the edge set replaces an edge $(u,v)$ with the two edges $(u,w)$ and $(w,v)$.  That is,
$\tilde V=V\cup \{w\}$ and $\tilde E= E\setminus\{(u,v)\}\cup\{(u,w),(w,v)\}$.
\end{definition}
\begin{theorem}[Kuratowski's Theorem, \cite{Kuratowski1930}]
A finite graph, $G$, is planar if and only if it does not contain a subgraph that is a subdivision of $K_5$ or $K_{3,3}$, where $K_5$ is the complete graph on 5 vertices and $K_{3,3}$ is the complete bipartite graph on six vertices (also known as the utility graph), see Figure \ref{fig:K5K33}.
\end{theorem}

A weaker formulation of Kuratowski's Theorem can be stated in terms of graph minors.

\begin{definition}
Given an undirected graph $G(V,E)$, consider edge $e=(u,v)\in E$.  \emph{Contracting} the edge $e$ entails deleting edge $e$ and identifying $u$ and $v$ as the same vertex.  The resulting graph $H(\tilde V,\tilde E)$ has one fewer edge and vertex as $G$.  

An undirected graph is called a \emph{minor} of $G$ if it can be formed by contracting edges of $G$.
\end{definition}

\begin{theorem}[Wagner's Theorem, \cite{Wagner1937}]
A finite graph is planar if and only if it does not have $K_5$ or $K_{3,3}$ as a minor.
\end{theorem}
Because of the importance of $K_5$ and $K_{3,3}$ in identifying non-planar graphs, there are referred to as \emph{forbidden minors}.

\begin{figure}[htbp]
\begin{center}
\includegraphics[scale=.4]{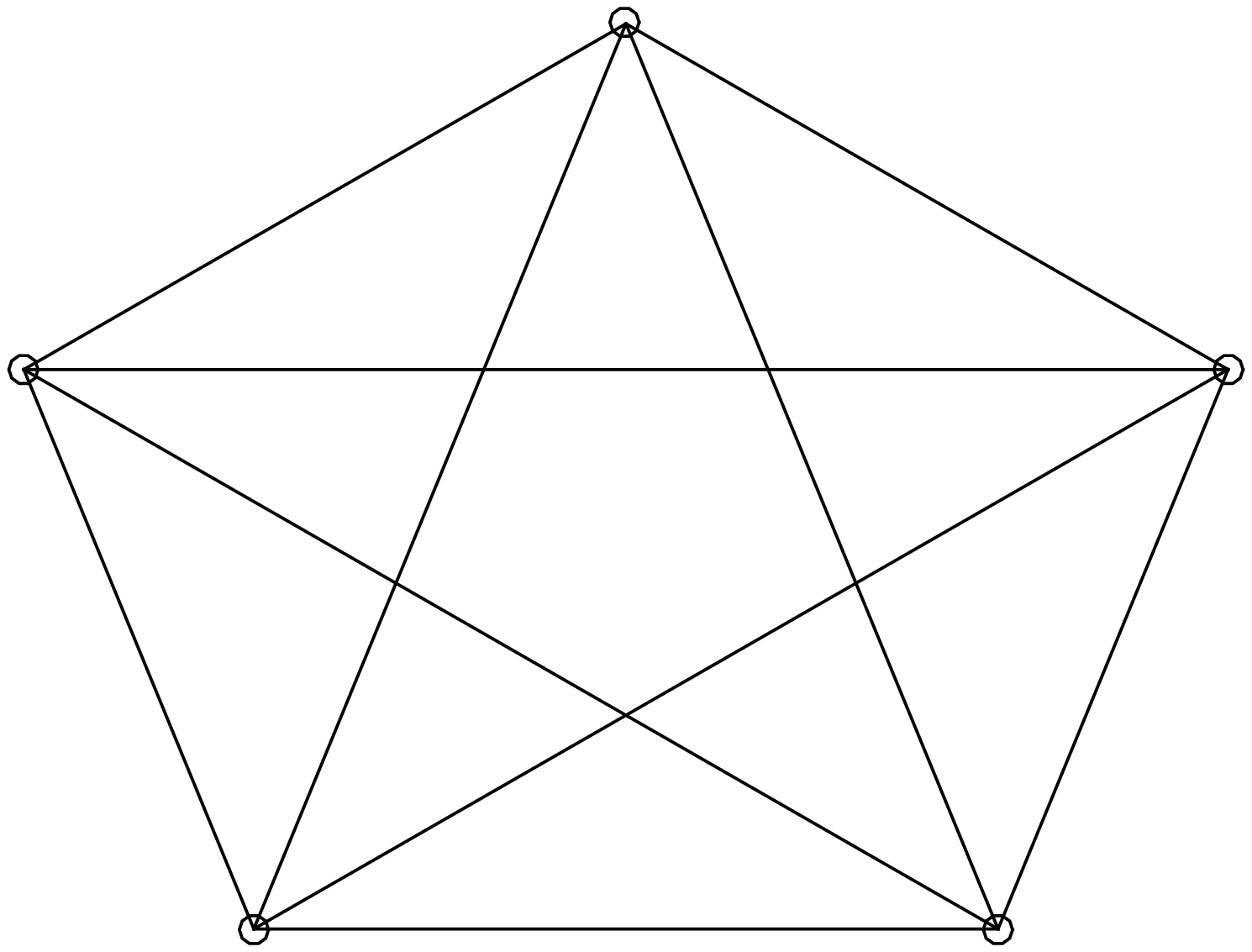}\includegraphics[scale=.4]{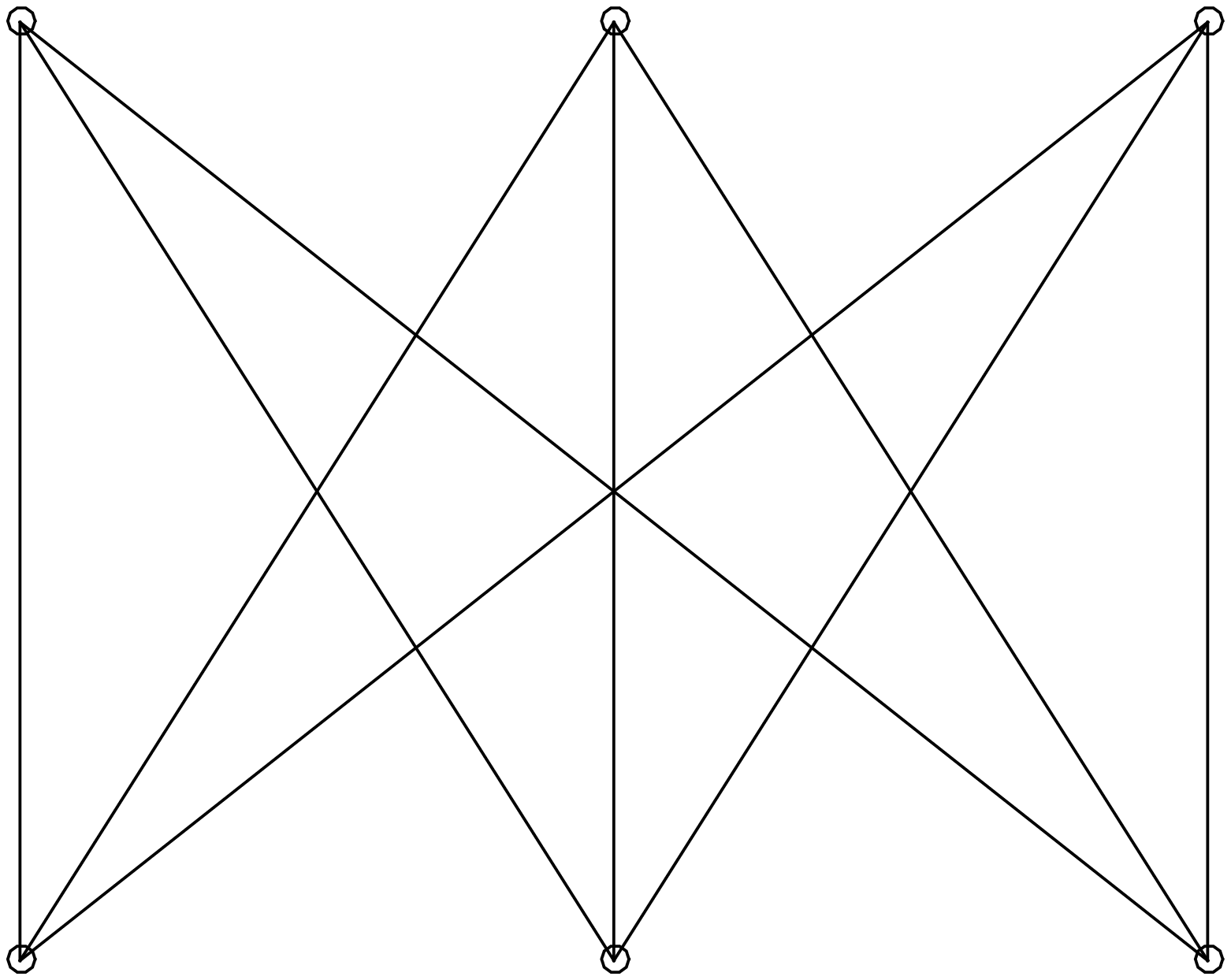}
\end{center}
\caption{The forbidden minors.  Left: The complete graph on five vertices.  Right: The complete bipartite graph on six vertices.}\label{fig:K5K33}
\end{figure}

One of the main results in this section shows that planar graphs cannot have large balls contained in the characteristic set $V_0$.
\begin{theorem}\label{th:planarballs}
Let $G(V,E)$ be a planar graph with Fiedler vector $\varphi_1$.  Then the zero set of $\varphi_1$ contains no balls of radius $r=1$ with more than three vertices. 
\end{theorem}
\begin{proof}
Suppose that $V_0$ contains a ball, $B_1(x)$, centered at vertex $x\in V_0$ and comprised of at least four vertices.  Without loss of generality, we can assume that the connected component of $V_0$ containing $x$ equals $B_1(x)$.  If not, then contract edges so that the connected component of $V_0$ containing $x$ equals a ball of radius 1.  Since $|B_1(x)|\geq 4$, then we have $d_x\geq 3$ and let $\{y_i\}_{i=1}^{d_x}$ denote the neighbors of $x$.  Then as constructed, $B_1(x)=\{x,y_1,y_2,...,y_{d_x}\}$.

By Lemma \ref{lem:Zset}, for $i=1,2,3$, each vertex $y_i$ has at least one neighbor in $V_+$ and at least one in $V_-$; pick one neighbor from each set and denote them $p_i$ and $n_i$, respectively.  It is proved in \cite{Urschel2014} that $V_+$ and $V_-$ are connected subgraphs of $G$.  Therefore, there is a path of edges that connect $p_1$, $p_2$, and $p_3$ (if $p_1=p_2=p_3$, then this path is empty).  We create a minor of $G$ by contracting the path connecting $p_1$, $p_2$, and $p_3$ to create one vertex $p\in V_+$.  Similarly, since $V_-$ is connected, we can contract the path connecting $n_1$, $n_2$, and $n_3$, to create one vertex $n\in V_-$.

Consider the subgraph of the now minorized version of $G$ consisting of vertices $\{x,p,n,y_1,y_2,y_3\}$.  This subgraph is $K_{3,3}$, the complete bipartite graph on six vertices since the vertices $\{x,p,n\}$ are all connected to $\{y_1,y_2,y_3\}$.  Thus by Wagner's Theorem, $G$ is not a planar graph, which is a contradiction.  This completes the proof.
\end{proof}

The result of Theorem \ref{th:planarballs} does not hold for general graphs.  We construct a family of (nonplanar) graphs for which $V_0$ contains a ball with a large number of vertices.  Since the set of vertices for which the Fiedler vector vanishes is large, we call this family of graphs the \emph{barren graphs}.  The barren graph with $|V|=N+7$ and $|V_0|=N+1$ is denoted Barr$(N)$.

\subsection{Construction of the barren graph, Barr$(N)$}\label{sec:barren}
The barren graph will be constructed as a sum of smaller graphs.
\begin{definition}\label{def:graphadd}
Let $G_1(V,E_1)$ and $G_2(V,E_2)$ be two graphs.  The \emph{sum} of graphs $G_1$ and $G_2$ is the graph $G(V,E)$ where $E=E_1\cup E_2$.
\end{definition}

The barren graph Barr$(N)$ is defined as follows
\begin{definition}
Let $K(V_i,V_j)$ denote the bipartite complete graph between vertex sets $V_i$ and $V_j$, that is, the graph with vertex set $V=V_i\cup V_j$ and edge set $E=\{(x,y): x\in V_i, y\in V_j\}$.  For $N\geq 3$ the barren graph, Barr$(N)$, is a graph with $N+7$ vertices.  Let $\{V_i\}_{i=1}^6$ denote distinct vertex sets with given cardinalities $\{|V_i|\}_{i=1}^6=\{N,1,2,2,1,1\}$.  Then the barren graph is the following graph sum of the 5 complete bipartite graphs 
$$\text{Barr}(N)=K(V_1,V_2)+K(V_1,V_3)+K(V_1,V_4)+K(V_3,V_5)+K(V_4,V_6).$$
\end{definition}

As constructed, Barr$(N)$ itself is bipartite; all edges connect the sets $V_2\cup V_3 \cup V_4$ to $V_1\cup V_5 \cup V_6$.  Figure \ref{fig:barren} shows two examples of barren graphs.

\begin{figure}[htbp]
\begin{center}
\includegraphics[scale=.5]{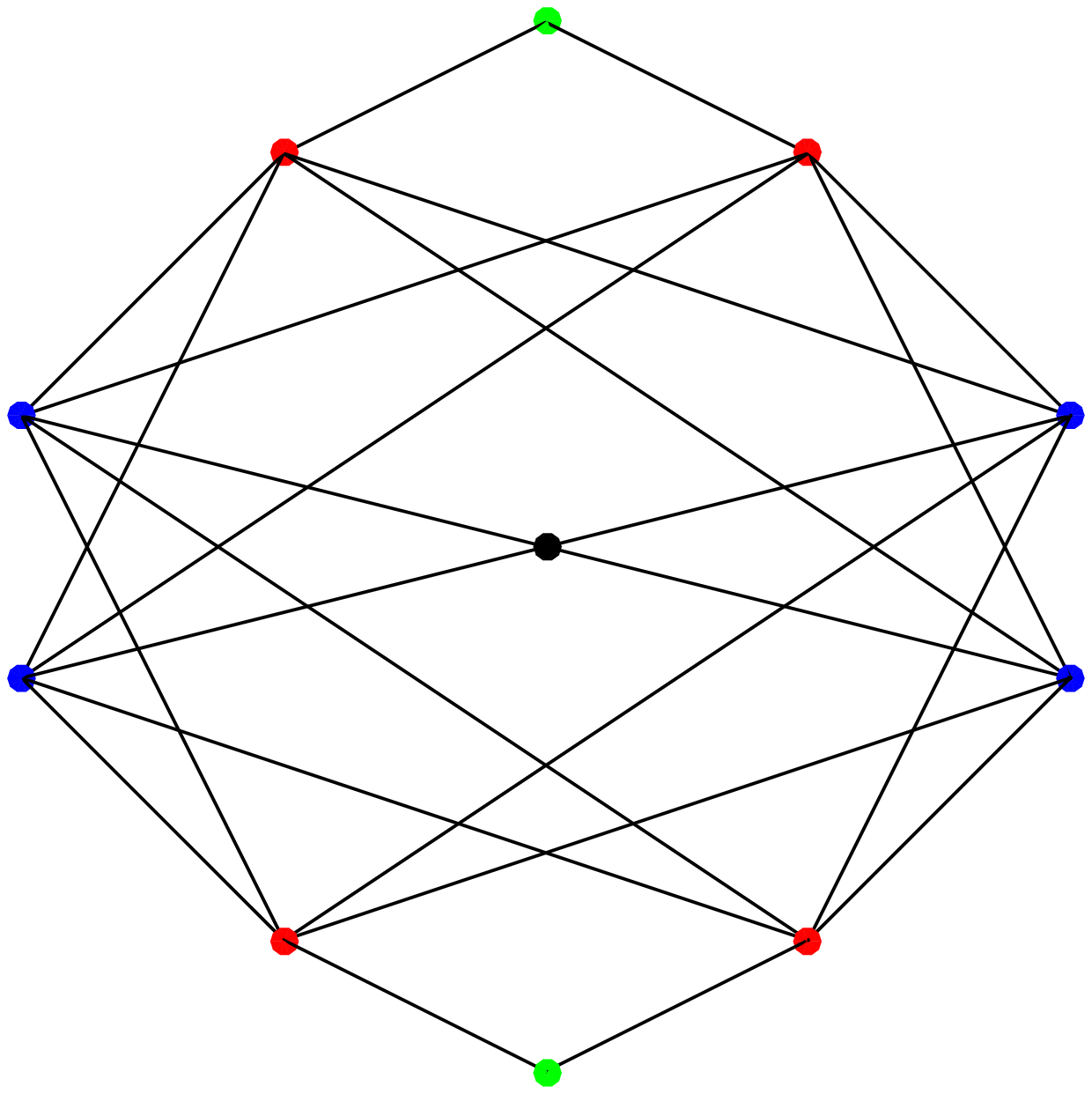}
\includegraphics[scale=.5]{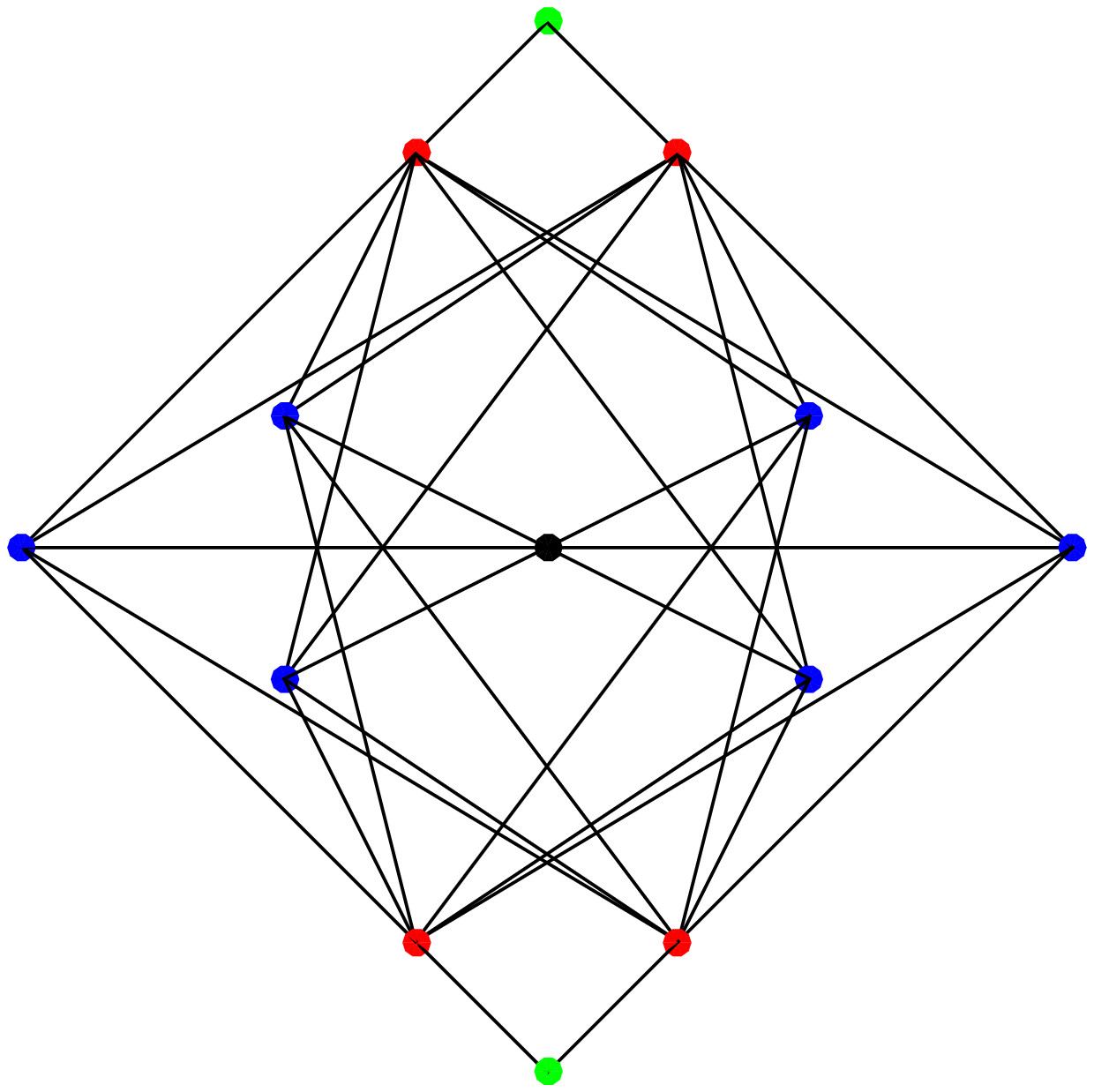}
\end{center}
\caption{The barren graph Barr$(4)$ (top) and Barr$(6)$ (bottom).  The set $V_1$ is denoted with $N$ blue dots; the vertex set $V_2$ is the black vertex in the center; the sets $V_3$ and $V_4$ are denoted with red dots; the sets $V_5$ and $V_6$ are denoted with green dots.}\label{fig:barren}
\end{figure}

We shall show that for any $N$, the Fiedler vector for Barr$(N)$ vanishes on $V_1 \cup V_2$ which has cardinality $N+1$.  Hence, the Fiedler vector for Barr$(N)$ has support on exactly six vertices for any $N\geq 3$.  In order to prove this, we explicitly derive the entire spectrum and all eigenvectors of the Laplacian.

\begin{theorem}\label{th:barren}
The barren graph, Barr$(N)$, has the spectrum given in Table \ref{tab:barreneigs}.  In particular, the Fiedler vector of Barr$(N)$ vanishes on vertices $V_1\cup V_2$ and hence $|\supp(\varphi_1)|=6$ for any $N$.
\end{theorem}

\begin{table}[htbp]
\begin{center}
\begin{tabular}{|c|c|c|}
\hline
$\lambda_k$ & value & eigenvector\\\hline\hline
$\lambda_0$ & 0 & constant function \\\hline
$\lambda_1$ & $\frac{1}{2} \left( N+3 - \sqrt{N^2-2N+9}\right)$ & Figure \ref{fig:barrenveca}\\\hline
$\lambda_2$ & $y_1$ & Figure \ref{fig:barrenvecb} \\\hline
$\lambda_3=\cdots= \lambda_{N+1}$& 5 & ON basis on $V_1$ \\\hline
$\lambda_{N+2}$& $y_2$ & Figure \ref{fig:barrenvecb}\\\hline
$\lambda_{N+3}=\lambda_{N+4}$& $N+1$ & Figure \ref{fig:barrenvecc} \\\hline
$\lambda_{N+5}$& $\frac{1}{2} \left( N+3 +\sqrt{N^2 -2N +9}\right)$ & Figure \ref{fig:barrenveca}\\\hline
$\lambda_{N+6}$& $y_3$ & Figure \ref{fig:barrenvecb} \\\hline
\end{tabular}
\caption{The spectrum of the barren graph, Barr$(N)$.  The values $y_1, y_2, y_3$ are the roots to the cubic polynomial \eqref{cubiceigs}.}\label{tab:barreneigs}
\end{center}
\end{table}

\begin{proof}

Firstly, the graph Barr$(N)$ is connected and so we have $\lambda_0=0$ with $\varphi_0\equiv (N+7)^{-1/2}$.  All other eigenvalues must be positive.

We will next show that the structure and support of the function shown in Figure \ref{fig:barrenveca} is an eigenvector for two eigenvalues of Barr$(N)$.

\begin{figure}[htbp]
\begin{center}
\begin{picture}(230,170)
\put(0,0){\includegraphics[scale=.4]{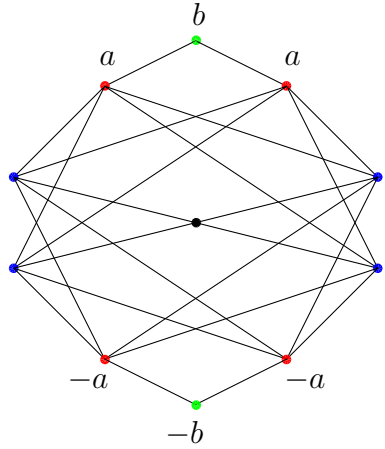}}
\put(115,162){$b$}
\put(80,147){$a$}
\put(150,147){$a$}
\put(68,25){$-a$}
\put(150,25){$-a$}
\put(105,4){$-b$}
\end{picture}
\end{center}
\caption{Support and function values for the eigenvectors associated with eigenvalues $\lambda_1$ and $\lambda_{N+5}$.}\label{fig:barrenveca}
\end{figure}
One can check upon inspection that the shown function $\varphi$ is orthogonal to the constant function.  Then, if the function shown in Figure \ref{fig:barrenveca}, call it $\varphi$, is an eigenvector, then the eigenvalue equation, $Lx=\lambda x$ is satisfied at each vertex.  It is easy to verify that $L\varphi(x) = 0$ for each $x\in V_1\cup V_2$.  For $x\in V_5$ or $x\in V_6$ the eigenvalue equation becomes $L \varphi(x) = 2(b-a) = \lambda b$.  For any $x\in V_3$ or $x\in V_4$, the eigenvalue equation gives $L\varphi(x)= Na+(a-b)=\lambda a$.  Finally, we also impose that the condition that the eigenvectors are normalized so that $\norm{\varphi}=1$.  Therefore, the function $\varphi$ shown in Figure \ref{fig:barrenveca} is an eigenvector of $L$ if and only if the following system of equations has a nontrivial solution:
\begin{equation*}
\left\{
\begin{array}{r c l}
4 a^2 +2 b^2 &=&1\\
2(b-a) &=& \lambda b\\
Na+(a-b)&=&\lambda a.
\end{array}
\right.
\end{equation*}

The first equation is not linear, but we can still solve this system by hand with substitution to obtain the following two solutions:
\begin{equation*}
\left\{
\begin{array}{r c l}
a&=& \frac{1}{2} \sqrt{\frac{N^2-2N+9 \mp (N-1)\sqrt{N^2-2N+9}}{2(N^2-2N+9)}}\\
b&=& \frac{1}{2} \sqrt{\frac{N^2-2N+9 \pm (N-1)\sqrt{N^2-2N+9}}{N^2-2N+9}}\\
\lambda &=&\frac{1}{2} \left( N+3 \pm \sqrt{N^2-2N+9}\right).
\end{array}
\right.
\end{equation*}
This gives two orthogonal eigenvectors and their eigenvalues.

Consider now the vector shown in Figure \ref{fig:barrenvecb} with full support, yet only taking on four distinct values.

\begin{figure}[htbp]
\begin{center}
\begin{picture}(230,170)
\put(0,0){\includegraphics[scale=.4]{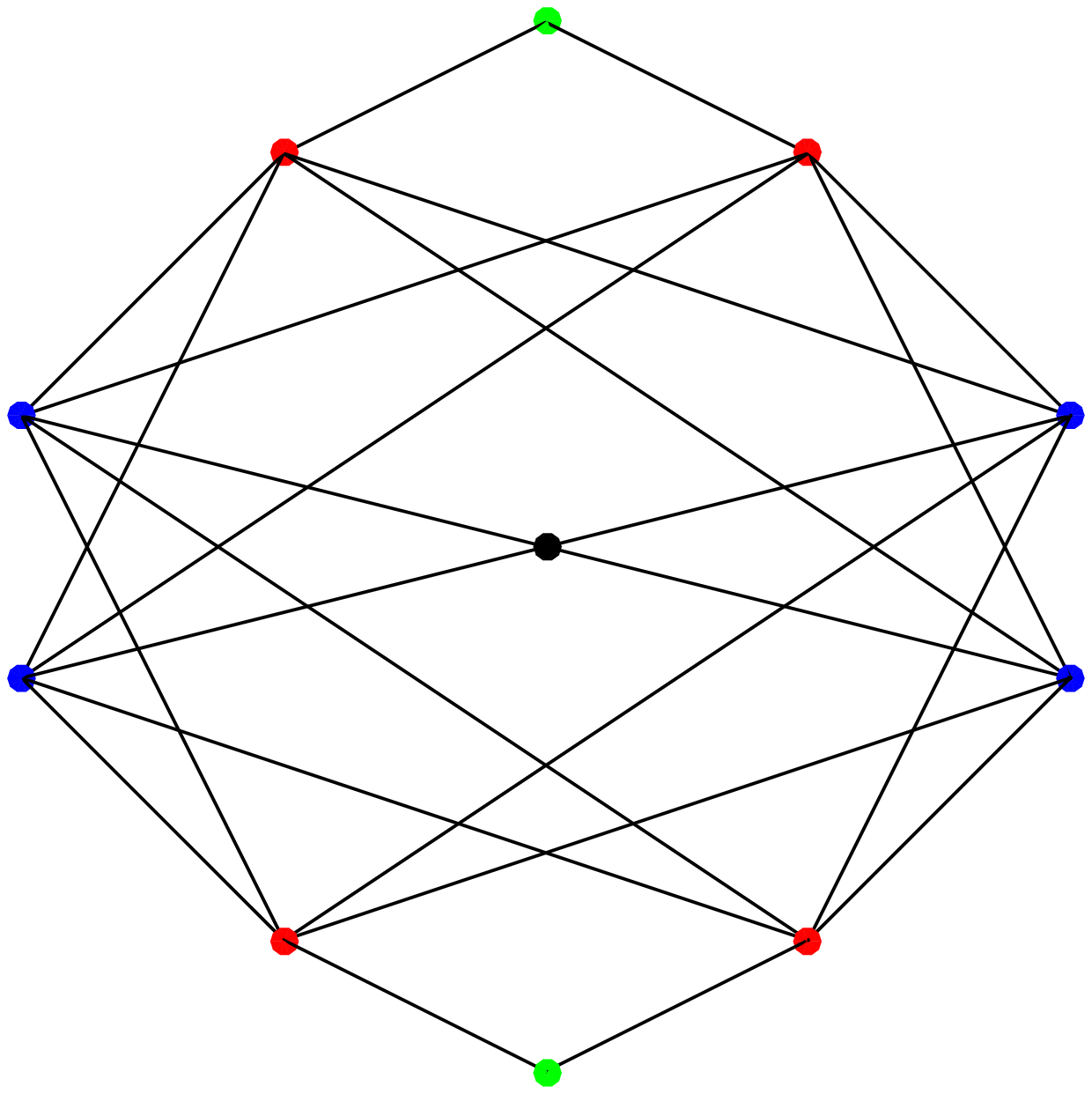}}
\put(115,162){$d$}
\put(80,147){$c$}
\put(150,147){$c$}
\put(80,25){$c$}
\put(150,25){$c$}
\put(115,4){$d$}

\put(115, 95){$b$}
\put(190, 105){$a$}
\put(190, 65){$a$}
\put(40, 105){$a$}
\put(40, 65){$a$}

\end{picture}
\end{center}
\caption{Support and function values for the eigenvectors associated with eigenvalues $\lambda_2$, $\lambda_{N+2}$, and $\lambda_{N+6}$.}\label{fig:barrenvecb}
\end{figure}
Similar to the previous example, we obtain a system of equations by imposing the conditions $\norm{\varphi}=1$, $\langle\varphi,1\rangle=0$, and from writing out the eigenvalue equations at each vertex class from $V_1, V_2, V_3$ and $V_5$ which gives:
\begin{equation*}
\left\{
\begin{array}{r c l r}
Na^2+b^2+4c^2 +2d^2 &=&1 & (\norm{\varphi}^2=1)\\
Na + b + 4c +2d &=& 0 & (\varphi \perp 1)\\
4(a-c)+(a-b)&=&\lambda a  &  (L\varphi(x)=\lambda\varphi(x): x\in V_1)\\
N(b-1)&=&\lambda b  &  (L\varphi(x)=\lambda\varphi(x): x\in V_2)\\
(c-d)+N(c-a)&=&\lambda c  &  (L\varphi(x)=\lambda\varphi(x): x\in V_3\cup V_4)\\
2(d-c)&=&\lambda d  &  (L\varphi(x)=\lambda\varphi(x): x\in V_5\cup V_6)
\end{array}
\right..
\end{equation*}
Again, this system cannot be solved with linear methods.  However, by tedious substitutions we can reduce the system (assuming each of the variables $a,b,c,d,\lambda$ are nonzero) to solving for the roots of the following cubic polynomial in $\lambda$:
\begin{equation}\label{cubiceigs}
\lambda^3 +(-2N-8)\lambda^2+(N^2+10N+15)\lambda + (-2N^2-14N) =0
\end{equation}
The cubic polynomial $x^3+c_2x^2+c_1 x+c_0=0$ has three distinct real roots if its discriminant, $\Delta=18 c_0c_1c_2 -4 c_2^3c_0 + c_2^2 c_1^2 -4 c_1^3 -27 c_0^2$, is positive.  The discriminant of \eqref{cubiceigs} is positive for all $N>0$ and hence we let $y_1<y_2<y_3$ denote the three positive roots which make up $\lambda_2, \lambda_{N+2}$, and $\lambda_{N+6}$, respectively.  By substituting back into the system of equations, one can obtain values for $a,b,c,d$ for each of the $\lambda=y_1,y_2,y_3$.  

The roots $y_1,y_2,y_3$ monotonically increase in $N$.  A simple calculation shows that $y_1=2$ for $N=3$ and $y_1>2$ for $N>3$.  Hence $\lambda_1 < \lambda_2 = y_1$ for all $N$.  Also observe that $y_2<5$ for $N<5$, so the ordering of the eigenvalues in Table \ref{tab:barreneigs} can vary but their values are accurate.

One can verify that the three eigenvectors obtained from Figure \ref{fig:barrenvecb} are linearly independent and orthogonal to each eigenvector derived so far.

\begin{figure}[htbp]
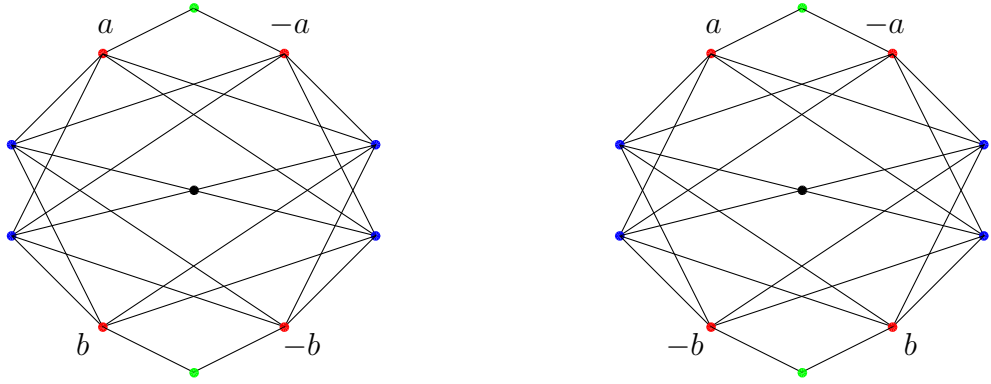

\begin{center}
\begin{picture}(460,170)
\put(0,0){\includegraphics[scale=.4]{counterexample4.eps}}
\put(80,147){$a$}
\put(145,147){$-a$}
\put(72,25){$b$}
\put(150,25){$-b$}

\put(230,0){\includegraphics[scale=.4]{counterexample4.eps}}
\put(310,147){$a$}
\put(370,147){$-a$}
\put(295,25){$-b$}
\put(385,25){$b$}

\end{picture}
\end{center}
\caption{Support and function values for the eigenvectors associated with eigenvalues $\lambda_{N+3}$ and $\lambda_{N+4}$.}\label{fig:barrenvecc}
\end{figure}
Consider now the two functions shown in Figure \ref{fig:barrenvecc}.  The eigenvalue equation gives $L\varphi(x)=0$ for every except for those $x\in V_3\cup V_4$ in which case we have $L\varphi(x)=(N+1)\varphi(x)$.  The two functions shown in Figure \ref{fig:barrenvecc} are orthogonal and linearly independent to each other and every eigenvector derived thus far and hence $N+1$ is an eigenvalue of Barr$(N)$ with multiplicity two.

Finally, we will construct eigenfunctions that are supported only on the $N$ vertices in $V_1$.  
Observe that if a function, $f$, is supported on $V_1$ then  for any $x\in V_1$, the eigenvalue equation gives $Lf(x)=5 f(x)$ since $x$ neighbors five vertices on which $f$ vanishes.  Therefore, Barr$(N)$ has eigenvalue 5.  To construct the corresponding eigenbasis, one can choose any orthonormal basis for the subspace of the $N$-dimensional vector space that is orthogonal to the constant vector.  Any basis for this $(N-1)$-dimensional vector space will give an eigenbasis for on $V_1$.  Finally one can verify by inspection that these $N-1$ eigenvectors are orthogonal and linearly independent to each eigenvector derived in this proof.

As such, we have now constructed an orthonormal, linearly independent eigenbasis for Barr$(N)$ corresponding to the eigenvalues given in Table \ref{tab:barreneigs}
\end{proof}

 As a remark, observe the behavior of the spectrum of Barr$(N)$ as $N\to\infty$.  For every natural number $N$, $\lambda_1<2$ and $\lim_{N\to\infty} \lambda_1=2$.  Using a symbolic solver, one can prove that $\lim_{N\to\infty}\lambda_2 = \lim_{N\to\infty} y_1=2$ as well.  Furthermore, the other two roots of the polynomial \eqref{cubiceigs} tend to infinity as $N\to\infty$.  Therefore, as $N\to \infty$, Barr$(N)$ has spectrum approaching $0$ (with multiplicity 1), $2$ (with multiplicity 2), $5$ (with multiplicity $N-2$), and the rest of the eigenvalues tending to $\infty$.

\subsection{Characteristic vertices and graph adding}\label{sec:graphadding}

In this subsection, we prove results about eigenvectors and their characteristic vertices for graph sums as defined in Definition \ref{def:graphadd}.  We borrow the following notation from \cite{Urschel2014} for clarity.
\begin{definition}
For any function $f$, Let 
$$i_0(f)=\{i\in V: f(i)=0\},$$
$$i_+(f)=\{i\in V: f(i)>0\},$$
$$i_-(f)=\{i\in V: f(i)<0\}.$$
\end{definition}
Observe that the set $V_0$ (resp. $V_+$ and $V_-$) from Section \ref{subsec3.1} is equal to $i_0(\varphi_1)$ (resp. $i_+(\varphi_1)$ and $i_-(\varphi_1)$).

\begin{theorem}\label{th:graphadd1}
Consider $n\geq 2$ connected graphs, $\{G_j(V_j,E_j)\}_{j=1}^n$ and suppose that all $n$ graph Laplacians, $L_j$, share a common eigenvalue $\lambda>0$ with corresponding eigenvectors $\varphi_{(j)}$.  Each graph's vertex set, $V_j$, assumes a decomposition $V_j=i_+(\varphi_{(j)})\cup i_-(\varphi_{(j)})\cup i_0(\varphi_{(j)})$ and suppose that $i_0(\varphi_{(j)})\neq\emptyset$ for all $j$.  Consider the graph $G(V,E)=G(\cup_{j=1}^n V_j, \cup_{j=1}^n E_j \cup E_0)$ where the edge set $E_0=\{(x_i,y_i)\}_{i=1}^K$ for $x_i \in i_0(\varphi_j)$, $y_i\in i_0(\varphi_\ell)$, and $j\neq \ell$ is nonempty.  Define $\varphi$ on $G$ by $\varphi(x)=\varphi_{(j)}(x)$ for $x\in V_j$.  Then, $\lambda$ is an eigenvalue of $G$ and $\varphi$ is a corresponding eigenvector.

Furthermore, if we add the assumption that the common eigenvalue $\lambda>0$ is the algebraic connectivity, i.e., the lowest nonzero eigenvalue of the graphs $G_j$, then $\lambda$ is an eigenvalue of $G(V,E)=G(\cup_{j=1}^n V_j, \cup_{j=1}^n E_j \cup E_0)$ but not the smallest positive eigenvalue.  Hence, $\varphi(x)$ is an eigenvector of $G$ but not its Fiedler vector.

\end{theorem}

\begin{proof}
We will verify that $L\varphi(x)=\lambda \varphi(x)$ for every $x\in V$.  Every $x\in V$ lies in exactly one $V_j$ and every edge connecting to $x$ must be in either $E_j$ or $E_0$.  Suppose $x$ contains no edges from $E_0$.  Then $L\varphi(x)=L_j\varphi_{(j)}(x)=\lambda \varphi_{(j)}(x)=\lambda \varphi(x)$.

Suppose instead that $x$ does contain at least one edge from $E_0$.  Then by construction of the set $E_0$, we have $\varphi(x)=0$ and $\varphi(y)=0$ for all $(x,y)\in E_0$.  This allows us to compute
\begin{eqnarray*}
 L\varphi(x)= \Sum_{y\sim x} \left(\varphi(x)-\varphi(y)\right) &=&    \Sum_{\substack{y\sim x\\(x,y)\in E_j}} \left(\varphi(x)-\varphi(y)\right)+\Sum_{\substack{y\sim x\\(x,y)\in E_0}}\left( \varphi(x)-\varphi(y)\right) \\
 &=& L_j \varphi_{(j)}(x) + 0 = \lambda \varphi_{(j)}(x) = \lambda \varphi(x).
 \end{eqnarray*}
Hence for any vertex in V, the vector $\varphi$ satisfies the eigenvalue equation and the first part of the proof is complete.

For the second claim of the theorem, let $G_1(V_1,E_2)$ and $G_2(V_2,E_2)$ have equal algebraic connectivities and Fiedler vectors $\varphi_{(1)}$ and $\varphi_{(2)}$, respectively.  We can decompose each vertex set into its positive, negative, and zero sets, i.e., $V_j=i_+(\varphi_{(j)})\cup i_-(\varphi_{(j)})\cup i_0(\varphi_{(j)})$.  Furthermore, $i_+(\varphi_{(j)})$ and $i_-(\varphi_{(j)})$ are connected subgraphs of $G_j$.

Now consider the larger graph $G(V,E)$.  The function $\varphi(x):= \varphi_{(j)}(x)$ for $x\in V_j$ is an eigenfunction of $G$ by the first part of the theorem.  However, now, the sets $i_+(\varphi)$ and $i_-(\varphi)$ are disconnected.  Indeed, let $x\in i_+(\varphi_{(1)})$ and $y\in i_+(\varphi_{(2)})$.  Then any path connecting $x$ and $y$ must contain an edge in $E_0$ since all $E_0$ contains all edges connecting $G_1$ to $G_2$.  And hence any path connect $x$ to $y$ will contain at least two vertices in $i_0(\varphi)$.  

Then by \cite{Urschel2014} since $i_+(\varphi)$ and $i_-(\varphi)$ are both disconnected, then $\varphi$ cannot be the Fiedler vector of $G$ and $\lambda$ is not the smallest nonzero eigenvalue.
\end{proof}

We can prove a stronger statement in the specific case where the graphs share algebraic connectivity, $\lambda_1$.

We can state a generalization of Theorem \ref{th:graphadd1} for eigenvectors supported on subgraphs of $G$.

\begin{theorem}\label{th:vertexzero}
Consider the graph $G(V,E)$.  Let $S\subseteq V$ and let $H(S,E_S)$ be the resulting subgraph defined by just the vertices of $S$.  Suppose that $\varphi^{(S)}$ is an eigenvector of $L_S$, the Laplacian of subgraph $H$, with corresponding eigenvalue $\lambda$.  If $E(S,V\setminus S) = E(i_0(\varphi^{(S)}), V\setminus S)$, that is, if all edges connecting graph $H$ to its complement have a vertex in the zero-set of $\varphi^{(S)}$, then $\lambda$ is an eigenvalue of $G$ with eigenvector 
$$\varphi(x)=\left\{ \begin{array}{l l}
\varphi^{(S)}(x)& x\in S\\
0 & x\notin S. \end{array} \right.$$
\end{theorem}

\begin{proof}
The proof is similar to the proof of Theorem \ref{th:graphadd1} in that we will simply verify that $L\varphi(x)=\lambda\varphi(x)$ at every point $x\in V$.  For any $x$ in the interior of $S$, then $L\varphi(x) = L_S \varphi^{(S)}(x) = \lambda \varphi^{(S)}(x)=\lambda \varphi(x)$.  For any $x$ in the interior of $V\setminus S$, then $L\varphi(x)=0$ since $\varphi$ vanishes at x and all of its neighbors.  For $x\in \delta(S)$ (recall $\delta(S)=\{x\in S : (x,y)\in E \text{ and } y\in V\setminus S\}$), we have 
\begin{eqnarray*}
 L\varphi(x)= \Sum_{y\sim x} \left(\varphi(x)-\varphi(y)\right) &=&    \Sum_{\substack{y\sim x\\y\in S}} \left(\varphi(x)-\varphi(y)\right)+\Sum_{\substack{y\sim x\\y\notin S}}\left( \varphi(x)-\varphi(y)\right) \\
 &=& L_S \varphi^{(S)}(x) + (0-0) = \lambda \varphi^{(S)}(x) = \lambda \varphi(x),
 \end{eqnarray*}
 where the term $(0-0)$ arises from the fact that $\varphi(y)=0$ since $y\notin S$ and since $(x,y)\in E$ then by assumption $x\in i_0(\varphi^{(S)})$ and hence $\varphi(x)=0$.  The same logic shows that $L\varphi(x)=0$ for $x\in \delta(V\setminus S)$.  Hence, we have shown that $L\varphi(x)=\lambda \varphi(x)$ for every possible vertex $x\in V$ and the proof is complete.
\end{proof}

Theorem \ref{th:vertexzero} is interesting because it allows us to obtain eigenvalues and eigenvectors of graphs by inspecting for certain subgraphs.  Furthermore since the eigenvector is supported on the subgraph, it is sparse and has a large nodal set.

\begin{example}
Consider the star graph $S_N(V,E)$ which is complete bipartite graph between $N$ vertices in one class $(V_A)$ and 1 vertex in the other $(V_B)$.  Let $S$ be the subgraph formed by any two vertices in $V_A$ and the one vertex in $V_B$.  Then the resulting subgraph on $S$ is the path graph on 3 vertices, $P_3$.  It is known that $P_3$ has Fiedler vector $\varphi^{(S)}=(\sqrt{2},0,-\sqrt{2})$ and eigenvalue $\lambda=1$.  Then by Theorem \ref{th:vertexzero}, the star graph $S_N$ has eigenvalue $\lambda=1$ with eigenvector supported on two vertices.  In fact, $S_N$ contains exactly ${N\choose 2}$ path subgraphs all of which contain the center vertex and have $\varphi^{(S)}$ as an eigenvector.  However, only $N-1$ of them will be linearly independent.  This method of recognizing subgraphs explains why $S_N$ has eigenvalue 1 with multiplicity $N-1$ and we have identified a set of basis vectors for that eigenspace.
\end{example}

\bibliographystyle{amsplain}
\bibliography{MBKO.bib}
\end{document}